\documentclass[12pt]{amsart}

\usepackage[bookmarks,colorlinks,breaklinks]{hyperref}  
\hypersetup{linkcolor=blue,citecolor=blue,filecolor=blue,urlcolor=blue}
\input{xy}
\xyoption{all}
\usepackage{graphicx}
\usepackage{color}
\usepackage{verbatim}
\usepackage{amssymb}

\newtheorem{lemma}[equation]{Lemma}
\newtheorem{theorem}[equation]{Theorem}
\newtheorem{conjecture}[equation]{Conjecture}
\newtheorem{cor}[equation]{Corollary}
\newtheorem{prop}[equation]{Proposition}

\newtheorem{question}[equation]{Question}

\theoremstyle{remark}
\newtheorem{example}[equation]{Example}
\newtheorem{counter-example}[equation]{Counter-example}

\newtheorem{remark}[equation]{Remark}

\newtheorem*{acknowledgments}{Acknowledgments}

\numberwithin{equation}{section}
\numberwithin{figure}{section}

\newcommand\C{{\mathbb C}}

\renewcommand\P{{\mathbb P}}

\newcommand\R{{\mathbb R}}
\newcommand\Z{{\mathbb Z}}
\newcommand\Q{{\mathbb Q}}
\newcommand\N{{\mathbb N}}

\newcommand\Gal{\operatorname{Gal}}
\newcommand\Aut{\operatorname{Aut}}

\newcommand\Qbar{\overline{\mathbb{Q}}}
\newcommand\Zbar{\overline{\mathbb{Z}}}

\newcommand{\cFbar}{\overline{\mathcal{F}}}

\DeclareMathOperator{\hhat}{{\widehat{h}}}

\newcommand\pol  {\mathrm{pol}}
\newcommand{\cB}{\mathcal{B}}
\newcommand{\cD}{\mathcal{D}}

\newcommand{\cF}{\mathcal{F}}
\newcommand{\cO}{\mathcal{O}}
\newcommand{\cM}{\mathcal{M}}
\newcommand{\cS}{\mathcal{S}}
\newcommand{\cR}{\mathcal{R}}
\newcommand{\cP}{\mathcal{P}}
\newcommand{\fp}{\mathfrak{p}}
\begin{document}

\title[Bounded height in families of dynamical systems]{Bounded height in families of dynamical systems}

\author{L.~De~Marco}
\address{
Laura DeMarco\\
Department of Mathematics\\
Northwestern University\\
2033 Sheridan Road\\
Evanston, IL 60208 \\
USA
}
\email{demarco@northwestern.edu}

\author{D.~Ghioca}
\address{
Dragos Ghioca\\
Department of Mathematics\\
University of British Columbia\\
Vancouver, BC V6T 1Z2\\
Canada
}
\email{dghioca@math.ubc.ca}

\author{H. Krieger}
\address{
Holly Krieger\\
Department of Pure Mathematics and Mathematical Statistics\\
 University of Cambridge \\
Cambridge CB3 0WB\\
UK
}
\email{hkrieger@dpmms.cam.ac.uk}

\author{K.~D.~Nguyen}
\address{
Khoa Dang Nguyen\\
Department of Mathematics\\
University of British Columbia\\
Vancouver, BC V6T 1Z2\\
Canada}
\email{dknguyen@math.ubc.ca}

\author{T.~J.~Tucker}
\address{
Thomas Tucker\\
Department of Mathematics\\ 
University of Rochester\\
Rochester, NY 14627\\
USA
}
\email{thomas.tucker@rochester.edu}

\author{H. Ye}
\address{
Hexi Ye\\
Department of Mathematics\\
Zhejiang University\\
Hangzhou, 310027\\
China}
\email{yehexi@gmail.com}

\date{\today}

\begin{abstract}
Let $a,b\in\Qbar$ be such that exactly one of $a$ and $b$ is an algebraic integer, and let $f_t(z):=z^2+t$ be a family of polynomials parametrized by $t\in\Qbar$. We prove that the set of all $t\in\Qbar$ for which there exist positive integers $m$ and $n$ such that $f_t^m(a)=f_t^n(b)$ has bounded height. This is a special case of a more general result supporting a new bounded height conjecture in dynamics. Our results fit into the general setting of the principle of unlikely intersections in arithmetic dynamics.
\end{abstract}

\maketitle

\thispagestyle{empty}


\section{Introduction}
\label{section introduction}
A subset of $\P^1(\Qbar)$ is said to have bounded height if the Weil height is bounded on this set.
To quickly illustrate the types of results and topics treated in this paper, we state the following theorem about the family of quadratic polynomials $\{z^2+t:  t\in\Qbar\}$, the most intensively family in complex and arithmetic dynamics:

\begin{theorem}\label{thm:z^2+t}
Let $f_t(z)=g_t(z)=z^2+t$ and $a,b\in \Qbar$ such that 
exactly one of $a$ and $b$ is an algebraic integer. Then
the set 
$$S:=\{t\in \Qbar:\ \text{$f_t^m(a)=g_t^n(b)$ for some
$m,n\in\N$}\}$$
has bounded height.
\end{theorem}

\noindent
For the conclusion of Theorem \ref{thm:z^2+t} to hold, some conditions on $a$ and $b$ are necessary:  if $a^2=b^2$, then $f_t^n(a)=g_t^n(b)$ for every $t$ and every $n\geq 1$, hence $S=\Qbar$.  However, we expect that the condition $a^2= b^2$ is the only obstruction to $S$ having bounded height.

The main goal of this paper is to formulate a web of bounded height results and questions, in the context of dynamical systems on $\P^1$, inspired by the results of Bombieri, Masser, and Zannier \cite{BMZ99, BMZ07} and more recent work such as in \cite{AMZ}.  Related results showing bounded height for dynamical systems have appeared in \cite{Ngu15,GN16}.  These questions, in turn, belong broadly to the principle of ``unlikely intersections'', as described in \cite{Zannier-book}.  In the dynamical context, see for example \cite{BD11, GKNY}, which specifically address unlikely intersections in the family of Theorem \ref{thm:z^2+t}.

Let $C$ be a smooth projective curve over $\Qbar$.  We fix a logarithmic Weil height $h_C$ on $C(\Qbar)$ associated to a divisor of degree one.  A subset of $C(\Qbar)$ is said to have bounded height if the function $h_C$ is bounded on this set; this is independent of the choice of $h_C$ \cite[Proposition~B.3.5]{HS00}.  We are interested in the following: 

\begin{question}\label{question:general}
Let $C$ be a projective curve defined over $\Qbar$ and $\cF=\Qbar(C)$.  Fix $r\geq 2$, and $f_1(z),\ldots,f_r(z)\in \cF(z)$ of degrees $d_1, \ldots, d_r \geq 2$.  Given points $c_1,\ldots,c_r\in \P^1(\cF)$, let $f_i^n(c_i)$ denote their iterates under $f_i$ in $\P^1(\cF)$.  Let $V$ be a hypersurface in $(\P^1)^r$, defined over $\cF$.  When can we conclude that the set
\begin{eqnarray*} 
 \{\; t\in C(\Qbar) & : &\mbox{ there exist } n_1,\ldots, n_r \geq 0 \mbox{ such that the specialization }\\  
 	&& ( f_1^{n_1}(c_{1}),\ldots,f_r^{n_r}(c_r))_t \mbox{ lies in }  V_t(\Qbar) \;\}
\end{eqnarray*}
has bounded height? 
\end{question}

\begin{example} \label{power maps}
The case when each $f_i$ is a power map $z\mapsto z^{\pm d_i}$ is treated by \cite[Theorem~1.2]{AMZ}.  Indeed, let $c_i \in \cF^*$ be multiplicatively independent modulo constants; this means that for every $(k_1, \ldots, k_r) \in\Z^r\setminus\{0\}$, we have $c_1^{k_1}\cdots c_r^{k_r} \notin \Qbar$.  Let $V$ be any hypersurface in $(\P^1)^r$.  Amoroso, Masser, and Zannier proved that the set 
	$$\{t\in\Qbar: (c_1^{k_1}, \ldots, c_r^{k_r})_t \in V_t(\Qbar) \mbox{ for some } k_1,\ldots, k_r \in\Z\}$$ 
has bounded height unless the $c_i$ satisfy a special geometric structure; namely, there exists a tuple $(k_1,\ldots,k_r)\in\Z^r\setminus\{0\}$ so that $(c_1^{k_1}, \ldots, c_r^{k_r})_t \in V_t$ for \emph{every} $t\in \Qbar$.  Applying their result to powers of the form $k_i = \pm d_i^{n_i}$ gives boundedness of height for the set of Question \ref{question:general} for the dynamical systems $f_i(z) = z^{d_i}$.
\end{example} 

\begin{example} \label{Lattes examples}
Suppose that $E$ is a non-isotrivial elliptic curve over $\cF$ and $\phi_1, \phi_2$ are endomorphisms of $E$ of degrees $>1$.  Let $f_1(z), f_2(z) \in \cF(z)$ be the associated Latt\`es maps on $\P^1$; that is, $f_i$ is the quotient of $\phi_i$ via the projection $\pi: E \to \P^1$ that identifies a point on $E$ with its additive inverse.  Fix points $P_1, P_2 \in E(\cF)$ that are linearly independent on $E$, and let $V$ be the diagonal in $\P^1\times\P^1$.  We then let $c_i = \pi(P_i)$ for $i = 1,2$.  The set in Question \ref{question:general} consists of points $t\in C(\Qbar)$ for which the specializations $P_{i,t}$ satisfy a relation 
	$$\phi_{1,t}^{n_1}(P_{1,t}) = \pm \phi_{2,t}^{n_2}(P_{2,t})$$
on $E_t$ for some $n_1, n_2 \geq 0$.  This set has bounded height, as a consequence of Silverman's specialization theorem \cite{Silverman:specialization}.  Indeed, the specializations $P_{i,t}$ satisfy a linear relation on $E_t$ if and only if $\det (\langle P_{i,t}, P_{j,t}\rangle) = 0$, where $\langle \cdot, \cdot\rangle$ is the canonical height pairing; we then apply \cite[III Corollary 11.3.1]{Sil94}.
\end{example}


In this paper, we focus on Question \ref{question:general} when $r=2$ and $V$ is the diagonal of $\P^1\times\P^1$.  We first exclude the cases where the maps are associated to an underlying group structure; this will take us out of the contexts that are traditionally addressed in the literature, such as those of Examples \ref{power maps} and \ref{Lattes examples}, but it also allows us to describe more easily the conditions that should guarantee bounded height.  

We will say that a function $f(z)\in \cF(z)$ of degree $d\geq 2$ is \emph{special} if it is conjugate by a M\"obius transformation in $\cFbar(z)$ to $z^{\pm d}$, $\pm T_d(z)$ or a Latt\`es map.  The Chebyshev polynomial $T_d$ is the unique polynomial in $\Q[z]$ satisfying $T_d(z+1/z)=z^d+1/z^d$, so it is a quotient of the power map $z^d$.  A rational function $f(z)\in \cFbar(z)$ of degree $d\geq 2$ is a Latt\`es map if there exist an elliptic curve $E$ over $\cFbar$ together with finite morphisms $\pi: E\rightarrow \P^1_{\cFbar}$ and $\phi: E\rightarrow E$ such that $\pi\circ \phi=f\circ \pi$.  

\begin{conjecture}\label{conj:diagonal}
Fix $f(z),g(z)\in \cF(z)$ with degrees at least $2$ and $a,b\in \P^1(\cF)$. 
Assume that at least one of $f$ and $g$ is not special.  Set
$$S:=\{t\in C(\Qbar):\ \text{$f_t^m(a(t))=g_t^n(b(t))$ for some $m,n \geq 0$}\}.$$ 
Then at least one of the following statements must hold:
\begin{itemize}
	\item [(1)] Either $(f,a)$ or $(g,b)$ is isotrivial.
	\item [(2)] There exist $m,n \geq 0$ such that
	$f^m(a)=g^n(b)$.
	\item [(3)] $S$ has bounded height.
\end{itemize}
\end{conjecture}

A pair $(f,c)$, with $f(z) \in \cF(z)$ and $c \in \P^1(\cF)$, is said to be \emph{isotrivial} if there exists a fractional linear transformation $\mu\in \cFbar(z)$ such that $\mu\circ f\circ \mu^{-1}\in \Qbar(z)$ and $\mu(c)\in \P^1(\Qbar)$.  Condition (2) is clearly an obstruction to $S$ having bounded height.  Condition (1) can also lead to unbounded height.  To see this, assume that $f(z)\in \Qbar(z)$ and $a\in \P^1(\Qbar)$ is such that $a$ is not preperiodic for $f$.  Then the sequence $\{f^m(a)\}_{m\geq 0}$ has unbounded height in $\P^1(\Qbar)$, by the Northcott property of the Weil height.  Fixing $n$, the solutions to the equations
	$$f^m(a)=g_t^n(b(t))$$
as $m$ goes to infinity will also have unbounded height.

\begin{remark}
One possible application of Conjecture~\ref{conj:diagonal} is to the
theory of iterated monodromy groups.  Given a field $L$ and a rational
function $\varphi(z) \in L(z)$ of degree at least 2, one obtains Galois
extensions $L_n$ of $L$ by considering the splitting field of
$\varphi^n(z) - u$ over $L(u)$ where $u$ is a transcendental;
equivalently, each field $L_n$ may be viewed as the Galois closure of
the extension of function fields induced by the map
$\varphi^n: \P^1_L \longrightarrow \P^1_L$ (see \cite{Jon13} and
\cite{Odo85} for surveys).  Passing
to the inverse limit of the Galois groups $\Gal(L_n/L)$ as $n$ goes to
infinity one obtains a group $G_\varphi$ and a natural map from
$G_\varphi$ to $\Aut(T_d)$, where $T_d$ is the infinite rooted $d$-ary
tree corresponding to inverse images of $u$ under iterates of
$\varphi$.  Pink \cite{Pin13} has shown that if $\varphi$ is quadratic
over a field of characteristic 0, then the map from $G_\varphi$ to
$\Aut(T_2)$ is surjective unless $\varphi$ is post-critically finite
or there is an $n$ such that $\varphi^n(a) = \varphi^n(b)$ for $a$,
$b$ the critical points of $\varphi$.  Hence Conjecture~\ref{conj:diagonal}, with $f = g$ a rational function of degree 2 and $a, b$ the critical points of $f$, implies that if the map from $G_f$ to $\Aut(T_2)$ is surjective, then for all $t \in C(\Qbar)$ outside of a set of bounded height,
the map from $G_{f_t}$ to $\Aut(T_2)$ is also surjective.  More
generally, combining the methods of \cite{JKMT16} with
Conjecture~\ref{conj:diagonal}, one might hope that within any
non-special one-parameter family $f$ of rational functions over $\Qbar$ with
degree $d \geq 2$, the image of the $G_{f_t}$ in $\Aut(T_d)$ is the
same for all $t$ outside a set of bounded height.
\end{remark}

Our next result allows us to produce many examples that satisfy Conjecture~\ref{conj:diagonal}.
Let $\hhat_f$ and $\hhat_g$ denote the canonical height functions on $\P^1(\cFbar)$ associated to dynamical systems $f$ and $g$ (see \cite{CS93}). Write $d_1=\deg(f)$ and $d_2=\deg(g)$; we say that $d_1$ and $d_2$
are multiplicative dependent if they have a common power. Define
$$\cM=\{(m,n)\in\N^2: \;  d_1^m\hhat_f(a)=d_2^n\hhat_g(b)\}$$
and 
$$S_{\cM}=\{t\in C(\Qbar): \; f_t^m(a(t))=g_t^n(b(t)) \mbox{ for some } (m,n)\in\cM \}.$$
Obviously, if $\cM$ is empty then $S_{\cM}$ is empty.  We have the following:

\begin{theorem}\label{thm:multiplicative dependence}
Let $C$, $\cF$, $f(z)$, $g(z)$, $a$, $b$, $S$, $\cM$, and $S_{\cM}$ be as above, and assume that $d_1=\deg(f)\geq 2$ and $d_2=\deg(g)\geq 2$ are multiplicatively dependent. 
(We allow the possibility that both $f$ and $g$ are special.) 
If the set $S_{\cM}$ has bounded height, then one of the following holds:
\begin{itemize}
	\item [(1)] Either $(f,a)$ or $(g,b)$ is isotrivial. 
	\item [(2)] There exist $m,n \geq 0$ such that
	$f^m(a)=g^n(b)$.
	\item [(3)] $S$ has bounded height. 
\end{itemize} 
In particular, if $\cM$ is empty, then Conjecture \ref{conj:diagonal} holds for the pairs $(f,a)$ and $(g,b)$.
\end{theorem}

\begin{example}\label{eg:cM is empty}
 $C=\P^1$, $\cF=\Qbar(t)$, $f(z)=z^4+t$, $g(z)=z^8+t$, $a=t+2017$, and $b=t^3+2018$.  Then
 $(f,a)$ and $(g,b)$ are not isotrivial. 
 For every $m,n\geq 0$, $f^m(a)$ is a polynomial of degree $4^m$, and $g^n(b)$
 is a polynomial
 of degree $3\times 8^n$.
Therefore $f^n(a)\neq g^m(b)$ for every $m,n\geq 0$.
Moreover
$\cM=\emptyset$ since $\displaystyle\frac{\hhat_g(b)}{\hhat_f(a)}=3$ is not a power of $2$.  By Theorem~\ref{thm:multiplicative dependence}, the set
$S$ has bounded height.
\end{example}

\begin{counter-example}
Assume $d_1$ and $d_2$ are multiplicatively independent. Consider $C=\P^1$, $\cF=\Qbar(t)$, $f(z)=z^{d_1}$, $g(z)=z^{d_2}$, $a=t$, $b=2t$. It is not hard to show that $(f,a)$ and $(g,b)$ are not isotrivial, and we have $f^m(a)\neq g^n(b)$ for every $m,n \geq 0$.  Moreover, we have $\hhat_f(a) = \hhat_g(b) = 1$ and so, $\cM$ is empty (since $d_1$ and $d_2$ are multiplicatively independent).  The set $S$ consists of elements of the form $2^{d_2^n/(d_1^m-d_2^n)}$ for $m,n\in\N$.  From our assumption on $d_1$ and $d_2$, the numbers
$\displaystyle\frac{d_1^m}{d_2^n}-1$ as $m,n\in\N$ can be arbitrarily close to $1$.  Hence $\vert d_2^n/(d_1^m-d_2^n)\vert$ can be arbitrarily large, and $S$ does not have bounded height.  This is a case ruled out by the hypothesis of \cite{AMZ} in Example \ref{power maps} above.  This also illustrates the exclusion of special maps from Conjecture \ref{conj:diagonal}.
\end{counter-example}

The proof of Theorem~\ref{thm:multiplicative dependence} is
given in Section~\ref{sec:useful reductions}. A key ingredient in
the proof is a well-known result by Call-Silverman \cite{CS93}. 
It seems much harder to prove Conjecture~\ref{conj:diagonal} when
$\cM\neq \emptyset$ as in Theorem~\ref{thm:z^2+t} (or the more
general Theorem~\ref{thm:general}). 
 
 A natural approach to proving
Theorems~\ref{thm:z^2+t} and \ref{thm:general} (or other cases of Conjecture \ref{conj:diagonal}) as well as the main result of
Amoroso-Masser-Zannier consists of two steps:
\begin{itemize}
\item [(i)] Let $K$ be a number field such that
$a,b\in K$.
For each $t\in S$, we construct a polynomial $P(x)\in K[x]$, depending on $f, g, a,b$, that vanishes at $t$ and whose degree is easy to compute.  We then obtain an upper bound on the height of the polynomial $P$. 
\item [(ii)] The second step is to prove that $t$ has a large degree over $K$, comparable to the degree of $P$.  This means that a certain factor of $P$ with large degree is irreducible over $K$.
\end{itemize}
While the first step is somewhat tedious, it only involves relatively straightforward height inequalities. However the second step is a notoriously hard problem in diophantine geometry.  Amoroso, Masser, and Zannier get around the second step by the construction of certain auxiliary polynomials using Siegel's lemma, the use of
Wronskians for certain zero estimates, and various careful height estimates. In our setting, we directly carry out the second step in this paper.  For certain examples, we can use basic tools (Eisenstein's criterion) to deduce irreducibility.  But for more interesting examples, such as the setting of Theorem \ref{thm:z^2+t}, we use the construction of $p$-adic B\"ottcher coordinates for families of polynomials.  This helps us relate $t\in S$ to a root of unity which automatically yields a very strong lower bound on the degree of $t$.  Our treatment of $p$-adic B\"ottcher coordinates extends earlier work of Ingram \cite{Ing13} in two important aspects.  First, it treats families of polynomials, hence is more flexible for applications to dynamics over parameter spaces.  Second, it allows the possibility that $p$ divides the degree of the polynomials in families. 
 
The organization of this paper is as follows. In the next section, we provide background on heights over number fields, function fields, and heights of polynomials following \cite{BG06,HS00}. This includes a well-known specialization theorem of Call-Silverman \cite{CS93} that plays an important role in the proof of Theorem~\ref{thm:multiplicative dependence}, which we give in Section~\ref{sec:useful reductions}.  In Section \ref{sec:upper bound} we provide upper bounds for heights of polynomials of the form $f^n(a)\in \Qbar[t]$ where $f(z)\in \Qbar[t][z]$.  Such upper bounds motivate a general approach to Theorem~\ref{thm:z^2+t} mentioned above and we give immediate examples based on Eisenstein's criterion in Section \ref{sec:Eisenstein}. In Section~\ref{sec:Bottcher}, we introduce non-archimedean B\"ottcher coordinates for families of polynomials, and we apply these in Section \ref{sec:proof} to prove Theorem~\ref{thm:general} which implies Theorem~\ref{thm:z^2+t}.

\begin{acknowledgments}
We thank the American Institute of Mathematics for its hospitality and support during a SQuaRE meeting in May 2016 which led to the present work. We are grateful also to Francesco Amoroso, David Masser and Umberto Zannier for their comments on a previous version of our paper.
\end{acknowledgments}



\bigskip
\section{Heights}
\label{sec:heights}

In this section we give background on heights.


\subsection{Heights over number fields}\label{subsec:heights Qbar}
Let $K$ be a number field and let
$M_K$ be the set of places of $K$. For each $\fp\in M_K$, let
$p\in M_{\Q}$ denote the restriction of $\fp$
to $\Q$, let $K_{\fp}$ be the completion of $K$ with respect to
$\fp$, and let $n_{\fp}=[K_{\fp}:\Q_p]$.
We define $\vert\cdot\vert_{\fp}$ to be the absolute value
on $K_{\fp}$ extending the standard absolute value
$\vert\cdot\vert_p$ on $\Q_p$ (see \cite[pp.~171]{HS00}). 
Let $\Vert \cdot\Vert_{\fp}=\vert \cdot\vert_\fp^{n_{\fp}}$
so that the product formula 
$\displaystyle\prod_{\fp\in M_K}\Vert x\Vert_{\fp}=1$ holds for every
$x\in K^*$. Let $r\in \N$, for $P=[x_0:\ldots:x_r]\in \P^r(K)$
define 
$H_K(P)=\displaystyle\prod_{\fp\in M_K} 
\max_{0\leq i\leq r}\Vert x_i\Vert_{\fp}$. 
For $P\in \P^r(\Qbar)$, pick a number field $L$ such that $P\in \P^r(L)$, then define
$H(P)=H_L(P)^{1/[L:\Q]}$. This is independent of the choice of $L$. Define $h(P)=\log(H(P))$. 
Finally, we have the height functions $H$ and $h$ on $\Qbar$ by
embedding $\Qbar\rightarrow \P^1(\Qbar)$. 
Let $\phi(z)\in \Qbar(z)$ with degree $d\geq 2$, define
$\hhat_{\phi}$ on $\P^1(\Qbar)$ by 
the formula:
$$\hhat_\phi(x)=\lim_{n\to \infty}\frac{h(\phi^n(x))}{d^n}.$$
The following will be used repeatedly:
\begin{lemma}\label{lem:basic hhat}
\item [(a)] There is a constant $c_0$ depending only on
$\phi$ such that $\vert\hhat_\phi(x)-h(x)\vert\leq c_0$
for every $x\in \P^1(\Qbar)$.
\item [(b)] $\hhat_\phi(\phi(x))=d\hhat_\phi(x)$ for every
$x\in\P^1(\Qbar)$.
\end{lemma}
\begin{proof}
See \cite[Chapter~3]{Sil07}.
\end{proof}

\subsection{Heights of polynomials over $\Qbar$}\label{subsec:heights polynomials}
Let $K$ be a number field. Let $P$ be a nonzero polynomial in $K[X_1,\ldots,X_n]$ written as
$$P=\sum_{(i_1,\ldots,i_n)} a_{i_1,\ldots,i_n}X_1^{i_1}\ldots X_m^{i_n}.$$
For every $\fp\in M_K$, define
$\vert P\vert_{\fp}=\max_{(i_1,\ldots,i_n)} \vert a_{i_1,\ldots,i_m}\vert_{\fp}$
and 
$\displaystyle\Vert P\Vert_{\fp}=\max_{(i_1,\ldots,i_n)} \Vert a_{i_1,\ldots,i_m}\Vert_{\fp}=\vert P\vert_{\fp}^{n_{\fp}}$. 
We will also use $\displaystyle\ell_{1,\fp}(P):=\sum_{(i_1,\ldots,i_n)} \vert a_{i_1,\ldots,i_n}\vert_\fp$.
Then we define 
$H_{\pol,K}(P)=\prod_{\fp\in M_K}\Vert P\Vert_{\fp}$
and $h_{\pol,K} (P)=\log(H_{\pol}(P))$  
As before, for every $P\in \Qbar[X_1,\ldots,X_n]\setminus \{0\}$,
choose a number field $L$ such that
$P\in L[X_1,\ldots,X_n]$
then 
define $H_{\pol}(P)=H_{\pol,L}(P)^{1/[L:\Q]}$
and $h_{\pol}(P)=\log (H_{\pol}(P))$ which is
the height of the point whose projective
coordinates are the coefficients of $P$. 
Write $M_K=M_K^0\cup M_K^{\infty}$ where $M_K^0$ (respectively $M_K^{\infty}$) is the set of finite (respectively infinite) places.
We have the following:
\begin{lemma}\label{lem:Gauss-Gelfond}
Let $K$ be a number field. 
\begin{itemize}
\item [(i)] Let $P,Q\in K[X_1,\ldots,X_n]\setminus\{0\}$
and $\fp\in M_K^0$, we have
$\vert PQ\vert_{\fp}=\vert P\vert_{\fp} \vert Q\vert_{\fp}$.
\item [(ii)] Let $v\in M_K^{\infty}$, let $P_1,\ldots,P_m\in K[X_1,\ldots,X_n]$ and
set $P=\prod_{i=1}^m P_i$. We have:
$$2^{-d}\prod_{i=1}^m \vert P_i\vert_v
\leq \vert P\vert_v\leq 2^d\prod_{i=1}^m \vert P_i\vert_v$$
where $d=\deg(P)$ is the total degree of $P$.
\item [(iii)] With the notation as in part (ii), we have:
$$-d\log 2+\sum_{i=1}^m h_{\pol}(P_i)
\leq h_{\pol}(P)\leq d\log 2+\sum_{i=1}^m h_{\pol}(P_i).$$
\end{itemize}
\end{lemma}
\begin{proof}
Part (i) is Gauss's lemma \cite[Lemma~1.6.3]{BG06}
while part (ii) is Gelfond's lemma \cite[Lemma~1.6.11]{BG06}.
Part (iii) follows from (i), (ii), the definition
of $h_{\pol}$, and
the identity $\sum_{v\in M_K^{\infty}} n_v=[K:\Q]$. 
\end{proof}

\begin{cor}\label{cor:roots}
Let $d,d'\in\N$. Let $P(X)\in \Qbar[X]$ be a polynomial
of degree $d$ and let $\alpha\in \Qbar$ be such that
at least $d'$ Galois conjugates of $\alpha$
are roots of $P(t)$. We have:
$$h(\alpha)\leq \frac{d\log 2+h_{\pol}(P)}{d'}.$$
\end{cor}
\begin{proof}
We may assume that $P$ is monic and write
$P(t)=\prod_{i=1}^d (t-\alpha_i)$.  We apply Lemma~\ref{lem:Gauss-Gelfond} for $P_i(t)=t-\alpha_i$
and note that there are at least $d'$ Galois conjugates
of $\alpha$ among $\alpha_1,\ldots,\alpha_d$.
\end{proof}

\subsection{Heights over function fields}\label{subsec:heights function fields}
As in Section~\ref{section introduction},  
let $C$ be a smooth projective curve over $\Qbar$,
let
$\cF=\Qbar(C)$, and fix a Weil height
$h_C$ on $C(\Qbar)$ associated to a divisor of degree one. As in Subsection~\ref{subsec:heights Qbar}, we can define
$M_{\cF}$ (with the extra condition that
absolute values are trivial on the field of constants $\Qbar$),
$n_{\fp}:=1$ for every $\fp\in M_{\cF}$, and
the height functions $H_{\cF}$
and $h_{\cF}$ on
$\P^r(\cFbar)$. For $\phi(z)\in \cFbar(z)$ with degree
$d\geq 2$, we can also define
$\hhat_{\phi}$
on $\P^1(\cFbar)$ and Lemma~\ref{lem:basic hhat}
(with the extra condition that $c_1$ depends on
$\cF$ and $\phi$)
remains valid. 

We say that $f(z) \in \cF(z)$ has good reduction over an open subset $U \subset C$ if $f$ induces a morphism $f:  U\times \P^1 \to \P^1$ over $\Qbar$, given by $(t, z) \mapsto f_t(z)$.  In particular, if $f$ has degree $d$, then the specialization will be a well-defined rational map $f_t: \P^1\to \P^1$ of degree $d$.

The following are crucial ingredients in
the next section:
\begin{prop}\label{prop:height function fields}
Let $f(z)\in \cF(z)$ with degree $d\geq 2$. Let $C'$ be a dense open Zariski subvariety of $C$ such that
$f$ has good reduction over $C'$. Let $a\in\P^1(\cF)$. We have:
\begin{itemize}
\item [(a)] There exist positive constants $c_1$ and $c_2$
depending only on $C$ and $f$ such that
$$\vert \hhat_{f_t}(x)-h(x)\vert\leq c_1 h_C(t) + c_2$$
for every $t\in C'(\Qbar)$ and $x\in \P^1(\Qbar)$.

\item [(b)] Regard $a$ as a morphism from $C$ to $\P^1$ and let $\deg(a)$ denote the degree of this morphism. Assume that $h_C$ is a height function corresponding to the divisor
$\frac{1}{\deg(a)}a^*\cO_{\P^1}(1)$. 
There is
a constant $c_3$ depending only on $C$ and $a$ such that
$\vert h(a(t))-\deg(a)h_C(t)\vert \leq c_3$ for every $t\in C(\Qbar)$.

\item [(c)] $$\lim_{h_C(t)\to\infty}\frac{\hhat_{f_t}(a(t))}{h_C(t)}
=\hhat_f(a).$$

\item [(d)] Assume that $a$ is not $f$-preperiodic. We have 
$\hhat_f(a)=0$ if and only if $(f,a)$ is isotrivial.
\end{itemize}
\end{prop}
\begin{proof}
For part (a), write $f(z)=\displaystyle\frac{P(z)}{Q(z)}$ where $P(z),Q(z)\in \cF[z]$ with $\gcd(P(z),Q(z))=1$. 
Removing finitely many points from $C'$ if necessary, we may assume that $$\gcd(P_t(z),Q_t(z))=1$$ so that
$f_t(z)=\displaystyle\frac{P_t(z)}{Q_t(z)}$ for every
$t\in C'(\Qbar)$.
Following \cite{HS11}, we can define the height of
$f_t$, denoted $\tilde{h}(f_t)$, to be the height of the point whose projective coordinates
are coefficients of $P_t(z)$ and $Q_t(z)$.  
Then \cite[Proposition~6]{HS11} gives that there
are positive constants $c_4$ and $c_5$ depending only on
$d$ such that
$\vert \hhat_{f_t}(x)-h(x)\vert\leq c_4\tilde{h}(f_t)+c_5$
for every $x\in \P^1(\Qbar)$ and $t\in C'(\Qbar)$. 
Since the coefficients of $P_t(z)$ and $Q_t(z)$ are obtained by evaluating at $t$ the functions from $\Qbar(C)$ which are the coefficients of $P(z)$ and $Q(z)$, we get that 
$\tilde{h}(f_t)\leq c_6 h_C(t)+c_7$
for some $c_6$ and $c_7$ that depend only on $C$ and $f$.
This finishes the proof.

Part (b) follows from \cite[Theorem~B.3.2]{HS00}. 
Part (c) is a well-known result of Call-Silverman \cite[Theorem~4.1]{CS93}. Part (d) follows from a result of
Baker \cite{Bak09}.
\end{proof}


\bigskip
\section{The proof of Theorem~\ref{thm:multiplicative dependence}}
\label{sec:useful reductions}
Throughout this section, let $C$, $\cF$, and $a,b\in \P^1(\cF)$
be as in Conjecture~\ref{conj:diagonal}. Fix a height
$h_C$ on $C(\Qbar)$ as before. Let $f(z),g(z)\in \cF(z)$
with $d_1=\deg(f)\geq 2$
and $d_2=\deg(g)\geq 2$. In this section, we \emph{allow} the possibility that both $f$ and $g$ are special. Let $S$ be the set defined in Conjecture~\ref{conj:diagonal}.  We start with an easy case:
 
\begin{prop}\label{prop:preperiodic case}
Assume that $a$ is preperiodic for $f$. Then one
of the following holds:
\begin{itemize}
\item [(i)] Either $(f,a)$ or $(g,b)$ is isotrivial.
\item [(ii)] There exist $m,n\in\N_0$ such that $f^m(a)=g^n(b)$.
\item [(iii)] $S$ has bounded height.
\end{itemize}
\end{prop}
\begin{proof}
We assume that conditions~(i)-(iii) do not hold and we will derive a contradiction.  
 The problem is easy when $b$ is $g$-preperiodic. Indeed $S$ is the set of $t\in C(\Qbar)$ satisfying finitely many equations; hence either $S$ is finite or one of those equations holds for every
$t$. 

Now assume that $b$ is not $g$-preperiodic. Since we assumed that condition~(i) does not hold, then Propositon~\ref{prop:height function fields} (for the pair $(g,b)$) yields that $\hhat_g(b)>0$. 

Let $(t_j)_{j\in\N}$ be in $S$ such that $h_C(t_j)\to \infty$ as $j\to\infty$. Since 
$a$ is $f$-preperiodic, after restricting to a subsequence of
$(t_j)_{j\in\N}$ if necessary, we have the following. There
exist $M\in\N$ and a sequence $(n_j)_{j\in\N}\subset \N$  
such that
for $\alpha=f^M(a)$ and 
for every $j$, we have
\begin{equation}
\label{new equation 1}
\alpha(t_j)=g_{t_j}^{n_j}\left(b(t_j)\right).
\end{equation}
Furthermore, since we assumed that condition~(ii) does not hold, then we may assume (perhaps at the expense of replacing $\{t_j\}$ by a subsequence) that $n_j\to \infty$ as $j\to\infty$. 
To avoid triple indices, let $\hhat_{(j)}$ denote the canonical
height on $\P^1(\Qbar)$ associated to
the rational function $g_{t_j}$. Applying $\hhat_{(j)}$ to equation~\eqref{new equation 1} and dividing by $h_C(t_j)$, we have
\begin{equation}
\label{new equation 2}
\frac{\hhat_{(j)}(\alpha(t_j))}{h_C(t_j)}=d_2^{n_j}\frac{\hhat_{(j)}(b(t_j))}{h_C(t_j)}
\end{equation}
for every $j$. By Proposition~\ref{prop:height function fields}:
\begin{equation}
\label{new equation 3}
\lim_{j\to\infty}\frac{\hhat_{(j)}(\alpha(t_j))}{h_C(t_j)}=\hhat_g(\alpha)<\infty .
\end{equation}
On the other hand, using the fact that $n_j\to\infty$ as $j\to\infty$ and also that $$\lim_{j\to\infty} \frac{\hhat_{(j)}(b(t_j))}{h_C(t_j)}=\hhat_g(b)>0,$$ 
we get
\begin{equation}
\label{new equation 4}
\lim_{j\to\infty}d_2^{n_j}\frac{\hhat_{(j)}(b(t_j))}{h_C(t_j)}=\infty.
\end{equation}
Equations \eqref{new equation 2}, \eqref{new equation 3} and \eqref{new equation 4} yield a contradiction.
\end{proof}

For the rest of this subsection, \emph{we further assume that $d_1$
and $d_2$ are multiplicatively dependent}. Define
$$\cM=\{(m,n)\in\N^2:\ d_1^m\hhat_f(a)=d_2^n\hhat_g(b)\}.$$
Also, we let $C'$ be a Zariski dense open subset of $C$ such that both $f$ and $g$ have good reduction at the points of $C'(\Qbar)$. We let  
$$S_{\cM}=\{t\in C'(\Qbar):\ \text{there exist
$(m,n)\in\cM$ such that $f_t^m(a(t))=g_t^n(b(t))$}\}$$ 
as in Theorem~\ref{thm:multiplicative dependence}. 

\begin{proof}[Proof of Theorem~\ref{thm:multiplicative dependence}]
We assume that both $(f,a)$ and $(g,b)$ are not isotrivial. 
We also assume that $S_{\cM}$ has bounded height and we need to
prove that either $S$ has bounded height or there exist
$m,n\in\N$ such that
$f^m(a)=g^n(b)$. 
If either $a$ is $f$-preperiodic or $b$ is $g$-preperiodic then
 Proposition~\ref{prop:preperiodic case} finishes our proof. 
So, from now on, assume that neither $a$ nor $b$ is preperiodic. 
By Proposition~\ref{prop:height function fields}~(d), we have
$\hhat_f(a)>0$ and $\hhat_g(b)>0$.

By Proposition~\ref{prop:height function fields}~(a), there
exist positive constants $c_8$ and $c_9$ depending only
on $C$, $f$, and $g$ such that
\begin{equation}\label{eq:hft hgt h}
\max\{\vert \hhat_{f_t}(x)-h(x)\vert,\vert \hhat_{g_t}(x)-h(x)\vert\}\leq c_8h_C(t)+c_9
\end{equation}
for every $t\in C'(\Qbar)$ and every $x\in \P^1(\Qbar)$.

Let $\delta\geq 2$ be an integer such that both $d_1$ and 
$d_2$ are powers of $\delta$. Since the set
$\{\delta^s:\ s\in\Z\}$ is discrete in $\R_{>0}$, there
is a positive
lower bound $c_{10}$ for the sets
$$\{\vert \hhat_f(a)-\delta^s\hhat_g(b)\vert:\ s\in\Z\}\setminus\{0\}\ \text{and}$$
$$\{\vert \hhat_g(b)-\delta^s\hhat_f(a)\vert:\ s\in\Z\}\setminus\{0\}.$$

Choose $\epsilon\in (0,c_{10}/3)$. By Proposition~\ref{prop:height function fields}~(c), there exist $c_{11}$ depending only on $C$, $f$, and $g$ such that for every
$t\in C'(\Qbar)$ with $h_C(t)\geq c_{11}$, we have:
\begin{equation}\label{eq:epsilon}
\max\left\{\left\vert\frac{\hhat_{f_t}(a(t))}{h_C(t)}-\hhat_f(a)\right\vert,
\left\vert\frac{\hhat_{g_t}(b(t))}{h_C(t)}-\hhat_g(b)\right\vert\right\}\leq \epsilon.
\end{equation}

Let $t\in S\setminus S_{\cM}$
and assume for the moment that $h_C(t)\geq c_{11}$. There exist
$m,n\in\N_0$ with $d_1^m\hhat_f(a)\neq d_2^n\hhat_g(b)$
and $f_t^m(a(t))=g_t^n(b(t))$.
This gives $h(f_t^m(a(t)))=h(g_t^n(b(t)))$ which together with \eqref{eq:hft hgt h} yield:
\begin{equation}
\vert \hhat_{f_t}(f_t^m(a(t)))-\hhat_{g_t}(g_t^n(b(t)))\vert
\leq 2c_8h_C(t)+2c_9.
\end{equation}
By properties of canonical heights, we have:
\begin{equation}\label{eq:d1md2n}
\vert d_1^m\hhat_{f_t}(a(t))-d_2^n\hhat_{g_t}(b(t))\vert
\leq 2c_8h_C(t)+2c_9.
\end{equation}
We consider the case $d_1^m\leq d_2^n$. Inequality 
\eqref{eq:d1md2n} yields:
\begin{equation}\label{eq:d1m leq d2n}
\left\vert\frac{d_1^m}{d_2^n}\frac{\hhat_{f_t}(a(t))}{h_C(t)}
-\frac{\hhat_{g_t}(b(t))}{h_C(t)}
\right\vert \leq \frac{2c_8}{d_2^n}+\frac{2c_9}{d_2^nh_C(t)}\leq
\frac{2c_8}{d_2^n}+\frac{2c_9}{d_2^nc_{11}}.
\end{equation}
Using $\displaystyle\frac{d_1^m}{d_2^n}\leq 1$ and 
inequality \eqref{eq:epsilon}, we have:
\begin{equation}\label{eq:-2epsilon}
\left\vert\frac{d_1^m}{d_2^n}\hhat_f(a)-\hhat_g(b)\right\vert-2\epsilon
\leq \left\vert\frac{d_1^m}{d_2^n}\frac{\hhat_{f_t}(a(t))}{h_C(t)}
-\frac{\hhat_{g_t}(b(t))}{h_C(t)}
\right\vert.
\end{equation}
The left-hand side of \eqref{eq:-2epsilon} is
at least $c_{10}-2\epsilon$ which is greater than $\epsilon$ due
to the choice of $c_{10}$ and $\epsilon$. Together with 
\eqref{eq:d1m leq d2n} and \eqref{eq:-2epsilon}, we have:
$\displaystyle\epsilon\leq \frac{2c_8}{d_2^n}+\frac{2c_9}{d_2^nc_{11}}$
which implies 
$\displaystyle d_2^n\leq \frac{1}{\epsilon}\left(2c_8+\frac{2c_9}{c_{11}}\right)$.

The case $d_1^m\geq d_2^n$ is treated by completely similar arguments. We have proved the following: if $t\in S\setminus S_{\cM}$ satisfies $h_C(t)\geq c_{11}$ then
$$\max\{d_1^m,d_2^n\}\leq \frac{1}{\epsilon}\left(2c_8+\frac{2c_9}{c_{11}}\right).$$
Note that there are only finitely many such pairs $(m,n)$. Hence 
such a $t$ satisfies one of finitely many equations. We conclude that either there are finitely many such $t$'s or one of those equations holds for every $t$. This finishes the proof.
\end{proof}

\bigskip
\section{Upper bounds on heights of polynomials}
\label{sec:upper bound}

In this section, we provide some of the technical ingredients on heights of polynomials needed for the proofs of Theorem \ref{thm:z^2+t} and other cases of Conjecture \ref{conj:diagonal}.

Fix $d\geq 2$ and let
$$P(z)=z^d+a_1z^{d-1}+\ldots+a_{d-1}z+a_d$$
be the generic monic polynomial 
of degree $d$ in $z$. For each $n\in\N$, write
$$P^n(z)=\sum_{i=0}^{d^n}A_{n,i}z^{d^n-i}$$
where $A_{n,i}\in\Z[a_1,\ldots,a_d]$. Note that $A_{n,0}=1$ for
every $n$. For each $p\in M_{\Q}$,
$n\in\N$, and $0\leq i\leq d^n$, our first goal is to give an upper bound
on the total degree $\deg(A_{n,i})$ and 
the maximum $\vert A_{n,i}\vert_p$ of the $p$-adic values of the coefficients of  $A_{n,i}\in \Z[a_1,\ldots,a_d]$ (see Subsection~\ref{subsec:heights polynomials}). 
We have:
\begin{prop}\label{prop:deg Ani}
For every $n\in\N$ and $0\leq i\leq d^n$, we have
$\deg(A_{n,i})\leq i$.
\end{prop}
\begin{proof}
The proposition holds when $n=1$, we now proceed by induction. 
We have:
$$P^{n+1}(z)=(P^{n}(z))^d+\sum_{j=1}^{d}a_j(P^n(z))^{d-j}=
\left(\sum_{i=0}^{d^n}A_{n,i}z^{d^n-i}\right)^d+\sum_{j=1}^{d}a_j\left(\sum_{i=0}^{d^n}A_{n,i}z^{d^n-i}\right)^{d-j}.$$

Let $0\leq k\leq d^{n+1}$. The coefficient of $z^{d^{n+1}-k}$ 
in $\displaystyle\left(\sum_{i=0}^{d^n}A_{n,i}z^{d^n-i}\right)^d$
is
\begin{equation}\label{eq:i1 to id}
\sum_{(i_1,\ldots,i_d)} A_{n,i_1}\cdots A_{n,i_d}
\end{equation}
where $\sum$ is taken over the tuples $(i_1,\ldots,i_d)$  
in $\{0,\ldots,d^n\}^{d}$
such that $(d^n-i_1)+\ldots+(d^n-i_d)=d^{n+1}-k$,
or equivalently $i_1+\ldots+i_d=k$. By the induction hypothesis,
the total degree of each term in \eqref{eq:i1 to id} is
$$\deg(A_{n,i_1}\ldots A_{n,i_k})\leq i_1+\ldots+i_d =k.$$

For $1\leq j\leq d$, the coefficient of 
$z^{d^{n+1}-k}$ in 
$a_j \displaystyle\left(\sum_{i=0}^{d^n}A_{n,i}z^{d^n-i}\right)^{d-j}$
is
\begin{equation}\label{eq:i1 to id-j}
\sum_{(i_1,\ldots,i_{d-j})} a_j A_{n,i_1}\ldots  A_{n,i_{d-j}}
\end{equation}
where $\sum$ is taken over the tuples $(i_1,\ldots,i_{d-j})$
in $\{0,\ldots,d^n\}^{d-j}$
such that $(d^n-i_1)+\ldots+(d^n-i_{d-j})=d^{n+1}-k$, or equivalently $i_1+\ldots+i_{d-j}=d^n(d-j)-d^{n+1}+k=k-d^nj$. By the induction hypothesis, the total degree of each term 
in \eqref{eq:i1 to id-j} is
$$\deg(a_kA_{n,i_1}\ldots A_{n,i_{d-j}})\leq 1+i_1+\ldots+i_{d-j}
=1+k-d^nj<k$$
since $j\geq 1$ and $d^n\geq 2$.
Overall, the coefficient of $z^{d^{n+1}-k}$ is a polynomial in the
$a_i$'s whose total degree is at most $k$. This finishes the proof.
\end{proof}

For the non-archimedean places $p\in M_{\Q}^0$, we have
the following estimates:
\begin{prop}\label{prop:Ani at prime p}
For $n\in \N$ and $1\leq i\leq d^n$, we have:
$$\vert A_{n,i}\vert_p\leq \min\{1,\vert d\vert_p^{n-i}\}.$$
\end{prop}
\begin{proof}
We proceed by induction on $n$. The case $n=1$ is immediate. Let 
$k\in\{1,\ldots,d^{n+1}\}$, since
$A_{n+1,k}\in \Z[a_1,\ldots,a_d]$ we have $\vert A_{n+1,k}\vert_p\leq 1$. It remains to show
$\vert A_{n+1,k}\vert_p\leq \vert d\vert_p^{n+1-k}$. When
$k\geq n+1$, this holds trivially since 
$\vert d\vert_p^{n+1-k}\geq 1$. From now on, we assume 
$k<n+1$, hence $k<d^n$ (since $d\ge 2$).  

As in the proof of Proposition~\ref{prop:deg Ani}, we
have:
$$P^{n+1}(z)=\left(z^{d^n}+\sum_{i=1}^{d^n}{A_{n,i}}z^{d^n-i}\right)^d+a_1\left(z^{d^n}+\sum_{i=1}^{d^n}{A_{n,i}}z^{d^n-i}\right)^{d-1}+\ldots$$

Since $k<d^n$, the coefficient of $z^{d^{n+1}-k}$ must
come solely from
$$\left(z^{d^n}+\sum_{i=1}^{d^n}{A_{n,i}}z^{d^n-i}\right)^d
=z^{d^{n+1}}+\sum_{\ell=1}^d\binom{d}{\ell} z^{d^n(d-\ell)}\left(\sum_{i=1}^{d^n}{A_{n,i}}z^{d^n-i}\right)^{\ell}.$$

For each $\ell\in\{1,\ldots,d\}$, the coefficient of
$z^{d^{n+1}-k}$ in $\displaystyle\binom{d}{\ell} z^{d^n(d-\ell)}\left(\sum_{i=1}^{d^n}{A_{n,i}}z^{d^n-i}\right)^{\ell}$
is
\begin{equation}\label{eq:coefficient at ell}
\binom{d}{\ell}\sum_{(i_1,\ldots,i_\ell)} A_{n,i_1}\ldots A_{n,i_\ell}
\end{equation}
where $\sum$ is taken over the tuples $(i_1,\ldots,i_\ell)$
in $\{1,\ldots,d^n\}^{\ell}$ such that
$d^n(d-\ell)+(d^n-i_1)+\ldots+(d^n-i_\ell)=d^{n+1}-k$, or
equivalently $i_1+\ldots+i_\ell=k$. For such a tuple $(i_1,\ldots,i_\ell)$, by Lemma~\ref{lem:Gauss-Gelfond} and the induction hypothesis, we have
\begin{equation}\label{eq:Ani at p and ell}
\left\vert\binom{d}{\ell} A_{n,i_1}\ldots A_{n,i_\ell}\right\vert_p \leq \left\vert\binom{d}{\ell}\right\vert_p
\vert d\vert_p^{n\ell-i_1-\ldots-i_\ell}= \left\vert\binom{d}{\ell}\right\vert_p
\vert d\vert_p^{n\ell-k}.
\end{equation}
The right-hand side of \eqref{eq:Ani at p and ell}
is equal to $\vert d\vert_p^{n+1-k}$ when $\ell=1$ and is at most
$\vert d\vert_p^{n+1-k}$ when $\ell\geq 2$. This finishes
the proof.
\end{proof}

\begin{remark}
The upper bound $\vert d\vert_p^{n-i}$ in Proposition~\ref{prop:Ani at prime p} is crucial for the construction of $p$-adic 
B\"ottcher coordinates when $p\mid d$.
\end{remark}

For the archimedean place of $M_\Q$, we have the following:
\begin{prop}\label{prop:Ani at archimedean v}
Write $M_{\Q}^{\infty}=\{v\}$, recall the notation $\ell_{1,v}(P)$ in Subsection~\ref{subsec:heights polynomials}. For every $n\in\N$
and $0\leq i\leq d^n$, we have: 
$$\vert A_{n,i}\vert_v\leq \ell_{1,v}(A_{n,i})\leq 2^i\binom{d^n}{i}< 4^{d^n}.$$
\end{prop}
\begin{proof}
A priori, it seems that expanding $P^{n+1}(z)$ and using Lemma~\ref{lem:Gauss-Gelfond} as in the proof of Proposition~\ref{prop:Ani at prime p} would not be enough to prove the proposition; the reason comes from the large factor $2^d$ in Lemma~\ref{lem:Gauss-Gelfond}~(b) (which corresponds to the extra
factor $2^{d^{n+1}}$ in our inductive step). However we can use the following simple trick.

The inequality $\vert A_{n,i}\vert_v\leq \ell_{1,v}(A_{n,i})$
is obvious from the definitions in Subsection~\ref{subsec:heights polynomials}. It remains to prove the other inequality.
Notice that all the polynomials $A_{n,i}\in \Z[a_1,\ldots,a_d]$
have \emph{non-negative} coefficients. Consider the polynomial:
$$\tilde{P}(z)=(z+2)^d-2=z^d+\tilde{a}_1z^{d-1}+\ldots+\tilde{a}_d$$ 
where $\tilde{a}_j\in\N$ for $1\leq j\leq d$. We have:
$$(z+2)^{d^n}-2=\tilde{P}^n(z)=z^{d^n}+\sum_{i=1}^{d^n}\tilde{A}_{n,i}z^{d^n-i}.$$

On the one hand $\tilde{A}_{n,i}=A_{n,i}(\tilde{a}_1,\ldots,\tilde{a}_d)\geq \ell_{1,v}(A_{n,i})$
since $\tilde{a}_j\geq 1$ for every $j$. 
On the other hand, we have $\tilde{A}_{n,i}= 2^i\binom{d^n}{i}$
if $1\leq i<d^n$ and $\tilde{A}_{n,d^n}=2^{d^n}-2$. 
In any case, we have $\tilde{A}_{n,i}\leq 2^i\binom{d^n}{i}<4^{d^n}$. This finishes the proof.
\end{proof}

We have the following application:
\begin{cor}\label{cor:height of polynomials}
Let $K$ be a number field. 
Let $f(z)=z^d+\alpha_1(t)z^{d-1}+\ldots+\alpha_d(t)\in K[t][z]$ and let $a(t)\in K[t]$; in particular,  
$f^n(a)\in K[t]$ for every $n\in\N$. 
\begin{itemize}
	\item [(a)] There exists a finite set of places $\cS\subset M_K$ and positive constants $c_{12}$ and $c_{13}$
	depending only on $K$, $f$, and $a$ such that the following hold.
	\begin{itemize}
		\item [(i)] $\cS$ contains $M_K^{\infty}$.
		\item [(ii)] For every $\fp\in M_K\setminus \cS$ and every
		$m\in\N_0$, we have $\vert f^m(a)\vert_\fp\leq 1$.
		\item [(iii)] For every $\fp\in \cS$ and $m \geq 0$, we have
		$\vert f^m(a)\vert_\fp\leq c_{12}^{d^m}$.
		\item [(iv)] For every $m\geq 0 $ such that $f^m(a)\neq 0$, we have
		$h_{\pol}(f^m(a))\leq c_{13}d^m$.
	\end{itemize}
	
	\item [(b)] Let $D\geq 2$, $g(z)=z^D+\beta_1(t)z^{D-1}+\ldots+\beta_D(t)\in K[t][z]$, and $b(t)\in K[t]$. There exists a positive constant $c_{14}$ depending only on $K$, $f$, $a$, $g$, and $b$ such that for every $m,n \geq 0 $ satisfying 
	$f^m(a)\neq g^n(b)$, we have
	$$h_{\pol}(f^m(a)-g^n(b))  \leq c_{14}\max\{d^m,D^n\}.$$ 
\end{itemize}
\end{cor}
\begin{proof}
Let $\cS$ be a finite subset of $M_K$ containing $M_K^\infty$
such that for every
$\fp\in M_K\setminus \cS$, the coefficients of $a(t)$ and the $\alpha_i(t)$'s are $\fp$-adic integers. We have:
$$f^m(z)=z^{d^m}+\sum_{i=1}^{d^m} A_{m,i}(\alpha_1(t),\ldots,\alpha_d(t))z^{d^m-i},$$
therefore
$$f^m(a)=a(t)^{d^m}+\sum_{i=1}^{d^m} A_{m,i}(\alpha_1(t),\ldots,\alpha_d(t))a(t)^{d^m-i}.$$
Lemma~\ref{lem:Gauss-Gelfond} and Proposition~\ref{prop:Ani at prime p} shows that $\vert f^m(a)\vert_\fp\leq 1$ for every
$\fp\in M_K\setminus \cS$. Hence $\cS$ satisfies (i) and (ii) of part (a). 

Let $c_{15}$ be a positive constant such that:
$$\max\{\vert a\vert_\fp,\vert\alpha_1\vert_\fp,\ldots,\vert \alpha_d\vert_\fp\}\leq c_{15}$$
for every $\fp\in \cS$. 
Let $\delta=\max\{\deg(a),\deg(\alpha_1),\ldots,\deg(\alpha_d)\}$.
If $\fp\in\cS$ is non-archimedean, Lemma~\ref{lem:Gauss-Gelfond} and Propositions~\ref{prop:deg Ani}~and~\ref{prop:Ani at prime p} give:
$$\vert f^m(a)\vert_\fp\leq c_{15}^{d^m}$$
for every $m \geq 0$.
If $\fp\in\cS$ is archimedean,  Lemma~\ref{lem:Gauss-Gelfond},
Proposition~\ref{prop:deg Ani}, and Proposition~\ref{prop:Ani at archimedean v} give:
$$\vert f^m(a)\vert_\fp\leq (d^m+1)2^{\delta d^m}4^{d^m}c_{15}^{d^m}$$
for every $m \geq 0$. This shows the existence of $c_{12}$ satisfying (iii) in part (a). 

From the definition of $h_{\pol}$ and the formula
$\displaystyle\sum_{\fp\in M_K,\fp\mid p} n_{\fp}=[K:\Q]$
for every $p\in M_\Q$, we deduce (iv) from (i), (ii), and (iii).

For part (b), we apply part (a) to the pair $(g,b)$. 
By extending $\cS$ and increasing $c_{12}$, we may assume that 
(ii) and (iii) hold for the data 
$(g,b,\cS,c_{12})$.  The desired upper bound on
$h_{pol}(f^m(a)-g^n(b))$
is obtained from the corresponding upper bounds for
$\vert f^m(a)-g^n(b)\vert_\fp$
for $\fp\in M_K$.
\end{proof}


\bigskip
\section{Examples using Eisenstein's criterion}
\label{sec:Eisenstein}

In this section, we prove some special cases of Conjecture \ref{conj:diagonal}.  We begin with a brief discussion of our strategy for proving Theorem \ref{thm:z^2+t}, and then we prove two propositions where the irreducibility step in the proof can be carried out by applying Eisenstein's criterion.

\subsection{The proof strategy for Theorem \ref{thm:z^2+t}}  \label{strategy}
Consider $f(z)=g(z)=z^2+t\in\Q[t][z]$ and $a,b\in \Qbar$.  Assume that $a^2 \not= b^2$.  We have $\hhat_f(a)=\hhat_g(b)=\frac{1}{2}$, hence both $(f,a)$ and $(g,b)$ are not isotrivial.  Also, becuase $a^2\neq b^2$, we have $f^m(a)\neq g^n(b)$ for every $m,n \geq 0$.  Indeed, as a polynomial in $t$, we have that $f^n(a)$ and $g^n(b)$ have both degree $2^{n-1}$ (for $n\ge 1$); so, if $f^m(a)=g^n(b)$, then it must be that $m=n$. On the other hand, the coefficient of $t^{2^{n-1}-1}$ in $f^n(a)$ (respectively in $g^n(b)$) is $2^{n-1}a^2$ (respectively $2^{n-1}b^2$); so, $f^m(a)\ne g^n(b)$ for any $m,n\in\N$.   By Theorem~\ref{thm:multiplicative dependence}, in order to prove that the set 
$$S=\{t\in\Qbar:\ \text{$f_t^m(a)=g_t^n(b)$ for some $m,n\in\N$}\}$$
has bounded height, it suffices to show that the set
$$S_{\cM}=\{t\in\Qbar:\ \text{$f_t^n(a)=g_t^n(b)$ for some $n\in\N$}\}$$
has bounded height. 

Let $K$ be a number field such that 
$a,b\in K$. By Corollary~\ref{cor:roots} and 
Corollary~\ref{cor:height of polynomials}, it suffices to show
that there exists a positive constant $c_{16}$ depending
only on $K$, $f$, $a$, and $b$ such that the following holds.
For every $t_0\in S$, if $N$ denotes the smallest positive integer
such that $f^N(a)(t_0)=g^N(b)(t_0)$, then 
$[K(t_0):K]\geq c_{16}2^N$.  Note that $\deg(f^N(a)-g^N(b))=2^N-1$.  Observing that $f^{N-1}(a)-g^{N-1}(b)$ divides $f^N(a)-g^N(b)$ for all $N$, we aim to prove that the polynomial $\displaystyle \frac{f^N(a)-g^N(b)}{f^{N-1}(a)-g^{N-1}(b)}\in K[t]$
is ``almost irreducible" over $K$.

\subsection{A variant of Theorem \ref{thm:z^2+t}}

\begin{prop}
Let $p$ be a prime and let $d>1$ be a power of $p$. Let 
$f(z)=g(z)=z^d+t\in \Q[t][z]$. Let $a,b\in\Qbar\cap \mathbb{Q}_p$ one of which is
a $p$-adic unit while the other one is in $p\, \Z_p$. Then
the set:
$$\{t\in\Qbar:\ \text{there exist $m,n\in\N$ such that $f^m_t(a_t)=g_t^n(b_t)$}\}$$
has bounded height.
\end{prop}

\begin{proof}
Exactly as discussed in \S\ref{strategy} for $d=2$, we have that $a^d\neq b^d$ implies $f^m(a)\neq g^n(b)$ for every $m,n \geq 0$. 
Therefore, we must only show that there exists a positive constant $c_{17}$ such that for every $N\geq 2$, every root $t_0$ of the polynomial
	$$\frac{f^N(a)-g^N(b)}{f^{N-1}(a)-g^{N-1}(b)}=\prod_{\zeta\neq 1,\zeta^d=1}(f^{N-1}(a)-\zeta f^{N-1}(b))\in\Qbar[t]$$
satisfies $[\Q(t_0):\Q]\geq c_{17}d^N$. 
	 
We will show that for every $d$-th root of unity $\zeta\neq 1$, the polynomial $f^{N-1}(a)-\zeta f^{N-1}(b)\in \Q_p[\zeta][t]$ is irreducible over the cyclotomic field $\Q_p(\zeta)$.  For every $c\in\Z_p$, as easy induction on $N$ yields that 
	 $$f^{N-1}(c)= t^{d^{N-2}}+t^{d^{N-3}}+\ldots+t+c^{d^{N-1}}+R_{N-1,c}(t)$$
where $R_{N-1,c}(t)\in pt\Z_p[t]$ with $\deg_t\left(R_{N-1,c}\right)<d^{N-2}$.

We have that $\lambda=1-\zeta$
is a uniformizer of $\Z_p[\zeta]$ (note that $d$ is a power of $p$). When $N\geq 2$, we have:
$$P(t):=f^{N-1}(a)-\zeta f^{N-1}(b)=(1-\zeta)t^{d^{N-2}}+\sum_{i=0}^{d^{N-2}-1} a_it^i+a^{d^{N-1}}-\zeta b^{d^{N-1}}$$
where $a_i\in \Z_p[\zeta]$ with $\lambda\mid a_i$
for every $i$. The polynomial $P(t)$ is irreducible
over $\Z_p[\zeta]$ since 
$t^{d^{N-2}}P(1/t)$ is Eisenstein (note that our hypothesis on $a$ and $b$ guarantees that $a^{d^{N-1}}-\zeta b^{d^{N-1}}$ is a $p$-adic unit). Hence
$[\Q_p(t_0):\Q_p(\zeta)]= d^{N-2}$
and this finishes the proof.
\end{proof}

\subsection{A second example with quadratic polynomials}
In our next example, $g(z)=z^2$ is special while $f \neq g$ is a quadratic polynomial.
\begin{prop}\label{prop:3z^2+5}
Let $f(z)=3z^2+5$, $g(z)=z^2$, $a=b=t \in \Q[t]$. The set 
$$S=\{t_0\in\Qbar:\ \text{$f^m(t_0)=g^n(t_0)$ for some $m,n\in\N$}\}$$
has bounded height.
\end{prop}

\begin{proof}
First we notice that the canonical heights $\hhat_f(a)$ and $\hhat_g(b)$ are both equal to $1$. So, by Theorem~\ref{thm:multiplicative dependence}, it suffices to prove that the set 
$$S_{\cM}=\{t_0\in\Qbar:\ \text{$f^n(t_0)=g^n(t_0)$ for some $n\in\N$}\}$$ 
has bounded height. Now, for every $n\in\N$, the leading coefficient of $P_n(t):=f^n(a)-g^n(b)\in \Q[t]$
is $3^{2^n-1}-1$ which is not divisible by $5$, while the constant term is congruent to $5$ modulo $25$, and the coefficients of the remaining terms are divisible by $5$. By Eisenstein's criterion,
$P_n$ is irreducible over $\Q$. 
By Corollary~\ref{cor:roots} and 
Corollary~\ref{cor:height of polynomials}, this proves that $S_{\cM}$ has bounded height thanks to Corollary~\ref{cor:height of polynomials} (exactly as in the discussion of \S\ref{strategy}).  By Theorem~\ref{thm:multiplicative dependence} we conclude that $S$ has bounded height as well.
\end{proof}


\medskip
\section{Non-archimedean B\"ottcher coordinates}
\label{sec:Bottcher}

In this section, we introduce $p$-adic B\"ottcher coordinates near infinity for a polynomial.  We use this analysis in our proof of Theorem \ref{thm:z^2+t} and its generalization in Section \ref{sec:proof}.  Compare the usual definition over the complex numbers in, e.g., \cite[Chapter~9]{Mil06}.  See also \cite{Ing13} in the non-archimedean setting.

Fix $d\geq 2$, let 
$P(z)=z^d+a_1z^{d-1}+\ldots+a_{d-1}z+a_d$,  
and write 
$$P^n(z)=\sum_{i=0}^{d^n}A_{n,i}z^{d^n-i}=z^{d^n}\left( 1+\sum_{i=1}^{d^n}\frac{A_{n,i}}{z^i}\right)$$
as in Section~\ref{sec:upper bound}. Let $\cP=\Q[a_1,\ldots,a_d]$
be the ring of polynomials in the $a_i$'s with rational coefficients and let $\cR=\cP ((1/z))$ be the ring of Laurent
series in $1/z$ with coefficients in $\cP$. 
Define 
$\nu$ on $\cR\setminus\{0\}$ by
letting $\nu(F)$ be the 
lowest power of $1/z$ that appears in $F$ (for example
$\nu(z+\frac{1}{z^5})=-1$). The subring $\cR_0$ of $\cR$ containing all power series $F$ such that $\nu(F)\ge 0$ is precisely $\cP[[1/z]]$. We have that $\cR_0$ is a complete
topological ring in which a basis of neighborhoods of $0$ is:
$$\frac{1}{z}\cR_0 \supset \frac{1}{z^2}\cR_0\supset \frac{1}{z^3}\cR_0\supset\ldots$$

If $\displaystyle\alpha\in \frac{1}{z}\cR_0$ 
and $m\in\N$, the series:
$$(1+\alpha)^{1/m}:=1+\frac{1}{m}\alpha+\frac{1}{2!}\frac{1}{m}\left(\frac{1}{m}-1\right)\alpha^2+\ldots$$
is a well-defined element in $\cR_0$ and its $m$-th power
is $1+\alpha$. For $n\in\N$, we define the series:
\begin{align}\label{eq:F_n}
\begin{split}
F_n&=z\left(\frac{P^n(z)}{z^{d^n}}\right)^{1/d^n}=z\left(1+\sum_{i=1}^{d^n}\frac{A_{n,i}}{z^i}\right)^{1/d^n}=z\left(1+\frac{1}{d^n}\left(\sum_{i=1}^{d^n}\frac{A_{n,i}}{z^i}\right)+\ldots\right)\\
&=:z+\sum_{j=0}^{\infty}\frac{B_{n,j}}{z^j}
\end{split}
\end{align}
where $B_{n,j}\in \cP$ for every $j\geq 0$. We
now compare $F_{n+1}$ with $F_n$. We have:
\begin{align*}
 P^{n+1}(z)&=(P^n(z))^d+a_1(P^n(z))^{d-1}+\ldots+a_{d-1}P^n(z)
+a_d\\
 &= \left(z^{d^n}\left(\frac{P^n(z)}{z^{d^n}}\right)\right)^d + a_1\left(z^{d^n}\left(\frac{P^n(z)}{z^{d^n}}\right)\right)^{d-1}+\ldots+a_d\\
 &= z^{d^{n+1}}\left(\left(\frac{P^n(z)}{z^{d^n}}\right)^d +\frac{a_1}{z^{d^n}}\left(\frac{P^n(z)}{z^{d^n}}\right)^{d-1}+\frac{a_2}{z^{2d^n}}\left(\frac{P^n(z)}{z^{d^n}}\right)^{d-2}
 +\ldots+\frac{a_d}{z^{d^{n+1}}}\right) 
\end{align*}

so that
\begin{align*}
F_{n+1}(z)&=z\left(\left(\frac{P^n(z)}{z^{d^n}}\right)^d +\frac{a_1}{z^{d^n}}\left(\frac{P^n(z)}{z^{d^n}}\right)^{d-1}+\frac{a_2}{z^{2d^n}}\left(\frac{P^n(z)}{z^{d^n}}\right)^{d-2}
 +\ldots+\frac{a_d}{z^{d^{n+1}}}\right)^{1/d^{n+1}}\\
 &=z\left(\left(\frac{P^n(z)}{z^{d^n}}\right)^d + E_n\right)^{1/d^{n+1}}
\end{align*}
where $E_n\in \cR_0$ with $\nu(E_n)=d^n$. 
Put $\alpha=\displaystyle\left(\left(\frac{P^n(z)}{z^{d^n}}\right)^d + E_n\right)^{1/d^{n+1}}$ and 
$\beta=\displaystyle \left(\frac{P^n(z)}{z^{d^n}}\right)^{1/d^n}$,
we have $\nu(\alpha-\zeta\beta)=0$ for every $d^{n+1}$-th root
of unity $\zeta\neq 1$. Therefore
\begin{align*}
\nu\left(\left(\left(\frac{P^n(z)}{z^{d^n}}\right)^d + E_n\right)^{1/d^{n+1}}-\left(\frac{P^n(z)}{z^{d^n}}\right)^{1/d^n}\right)&=\nu(\alpha-\beta)=\nu(\alpha^{d^{n+1}}-\beta^{d^{n+1}})\\
&=\nu(E_n)=d^n.
\end{align*}
Therefore $F_{n+1}-F_n\in \displaystyle\frac{1}{z^{d^n-1}}\cR_0$. Hence the sequence $\{F_n\}_n$ converges in $\cR$ to 
a series:
$$\cB(z)=z+\sum_{j=0}^{\infty}\frac{B_j}{z^j}$$
where $B_j\in \cP$ for every $j\geq 0$. Since $F_{n+1}-F_n\in\displaystyle\frac{1}{z^{d^n-1}}\cR_0$, we have:
\begin{equation}\label{eq:compare B_j and B_nj}
B_j=B_{n,j}\ \text{if $j<d^n-1$.}
\end{equation}

For every monic polynomial $Q(z)\in \cP[z]\setminus \cP$ (i.e. $\deg(Q)\geq 1$), we have that 
$1/Q(z)$ belongs to $\cR_0$ and 
$\nu(1/Q(z))=\deg(Q)$. Therefore, for every 
series $F(z)=\displaystyle\sum_{i=-m}^{\infty}c_i/z^i\in \cR$,
the element
$F\circ Q(z)= F(Q(z)):=\displaystyle\sum_{i=-m}^{\infty} c_i/Q(z)^i$
is a well-defined element of $\cR$. From \eqref{eq:F_n},
we have:
$$F_n(P(z))=F_{n+1}(z)^d.$$
Together with the definition of $\cB$, we have:
\begin{equation}\label{eq:Bottcher functiontal equation}
\cB(P(z))=\cB(z)^d.
\end{equation}

For each $w\in M_\Q^0$, let $\C_w$ denote the completion
of $\Qbar_w$ and we use the same notation $\vert \cdot\vert_w$ to 
denote its extension on 
$\C_w$. Our goal is to provide a domain
$\cD\subset \C_w^{d+1}$
such that the series 
$\cB$ is convergent at every
$(z,a_1,\ldots,a_d)\in \cD$.

We need the following:
\begin{lemma}\label{lem:1-im and k!}
Let $k,m\in\N$ and let $p$ be a prime not dividing $m$. Then
$\displaystyle \frac{\prod_{i=0}^{k-1} (1-im)}{k!}$
is a $p$-adic integer.
\end{lemma}
\begin{proof}
%
One can obtain this result by simply counting the exponent of the prime $p$ in both the numerator and the denominator of the above fraction, but one can also use the following clever observation suggested by David Masser. The binomial coefficient $B_k(x)$ given by 
$$x\mapsto \frac{x\cdot (x-1)\cdots (x-k+1)}{k!}$$
sends $\mathbb{Z}$ into itself and so, since $\mathbb{Z}$ is dense in $\mathbb{Z}_p$, then it also sends $\mathbb{Z}_p$ into itself. Since $p\nmid m$, then $\frac{1}{m}\in\mathbb{Z}_p$ and so, $B_k(1/m)\in\mathbb{Z}_p$; in particular, $m^kB_k(1/m)$ is a $p$-adic integer, as desired.
\end{proof}

\begin{theorem}\label{thm:Bottcher}
Let $w\in M_\Q^0$.
\begin{itemize}
\item If $w\in M_\Q^0$ corresponds to a prime which does not divide $d$, let
$$\cD:=\left\{(z,a_1,\ldots,a_d)\in \C_w^{d+1}:\ 
\max\{1,\vert a_1\vert_w,\ldots,\vert a_d\vert_w\}<\vert z\vert_w\right\}.$$

\item If $w\in M_\Q^0$ corresponds to a prime $p\mid d$, let
$$\cD:=\left\{(z,a_1,\ldots,a_d)\in \C_w^{d+1}:\ 
\frac{\max\{1,\vert a_1\vert_w,\ldots,\vert a_d\vert_w\}}{\vert d\vert_w}p^{1/(p-1)}<\vert z\vert_w\right\}.$$
\end{itemize}
Then the following hold:
\begin{itemize}
	\item [(a)] For every $(z,a_1,\ldots,a_d)\in \cD$, the
	series
	$$z+\sum_{j=1}^{\infty} \frac{B_j(a_1,\ldots,a_d)}{z^j}$$
	is convergent. This defines a function 
	$\tilde{\cB}:\ \cD\rightarrow \C_w$. Moreover, if
	$z,a_1,\ldots,a_d$ belong to a finite extension
	$\kappa$ of $\Q_w$ then $\tilde{\cB}(z,a_1,\ldots,a_d)\in\kappa$.
	\item [(b)] For every $(z,a_1,\ldots,a_d)\in \cD$:
	$$\tilde{\cB}(z^d+a_1z^{d-1}+\ldots+a_d,a_1,\ldots,a_d)=\tilde{\cB}(z,a_1,\ldots,a_d)^d.$$
	\item [(c)] If $\tilde{\cB}(z,a_1,\ldots,a_d)=\tilde{\cB}(z',a_1,\ldots,a_d)$ then $z=z'$.
\end{itemize}
\end{theorem}
\begin{proof}
Part (b) follows from \eqref{eq:Bottcher functiontal equation}. 
We will prove parts (a) and (c) for the case $w\nmid d$ first.

Let $j\in \N_0$ and choose $n:=n(j):=\lceil \log_d(j+2)\rceil$
so that $j<d^n-1$ and $B_j=B_{n,j}$ by \eqref{eq:compare B_j and B_nj}. By \eqref{eq:F_n}, 
$B_{n,j}$ is the coefficient of $1/z^{j+1}$ in
$$\frac{1}{d^n}\left(\sum_{i=1}^{d^n}\frac{A_{n,i}}{z^i}\right)
+\frac{1}{2!}\frac{1}{d^n}\left(\frac{1}{d^n}-1\right)\left(\sum_{i=1}^{d^n}\frac{A_{n,i}}{z^i}\right)^2+\ldots$$

For each $k\in\N$ with $k\leq j+1$, let $c_{n,j,k}$ be the coefficient of
$1/z^{j+1}$
in
$$\frac{1}{k!}\frac{1}{d^n}\ldots \left(\frac{1}{d^n}-k+1\right)\left(\sum_{i=1}^{d^n}\frac{A_{n,i}}{z^i}\right)^k.
$$
We have:
\begin{equation}\label{eq:cnjk}
c_{n,j,k}=\frac{1}{k!}\frac{1}{d^n}\ldots \left(\frac{1}{d^n}-k+1\right)\sum_{(i_1,\ldots,i_k)} A_{n,i_1}\ldots A_{n,i_k}
\end{equation} 
where $\sum$ is taken over the tuples $(i_1,\ldots,i_k)$
in $\{1,\ldots,d^n\}$ such that
$i_1+\ldots+i_k=j+1$.
Let $(z,a_1,\ldots,a_d)\in \cD$. Write $M=\max\{1,\vert a_1\vert_w,\ldots,\vert a_d\vert_w\}$.

If $w\in M_\Q^0$ with
$w\nmid d$, from Proposition~\ref{prop:deg Ani}, Proposition~\ref{prop:Ani at prime p}, and Lemma~\ref{lem:1-im and k!}, we have:
\begin{equation}\label{eq:w nmid d first}
\vert B_{n,j}\vert_w=\left\vert\displaystyle\sum_{k=1}^{j+1} c_{n,j,k}\right\vert_w\leq M^{j+1},
\end{equation}
\begin{equation}\label{eq:w nmid d second}
\text{and hence }\left\vert\frac{B_{n,j}}{z^j}\right\vert_w\leq M\left(\frac{M}{\vert z\vert_w}\right)^{j}.
\end{equation}

Therefore for every $(z,a_1,\ldots,a_d)\in\cD$, the series
$$z+\sum_{j=0}^{\infty}\frac{B_j(a_1,\ldots,a_d)}{z^j}$$
is convergent. The last assertion in part (a) follows from the
completeness of $\kappa$. For part (c),
we have:
\begin{equation}\label{eq:c non-arch w 1}
0=\tilde{\cB}(z,a_1,\ldots,a_d)-\tilde{\cB}(z',a_1,\ldots,a_d)
= (z-z')+\sum_{i=1}^{\infty}\frac{B_i(a_1,\ldots,a_d)(z'^i-z^i)}{z^iz'^i}.
\end{equation}
Assume $z\neq z'$ and we arrive at a contradiction as follows. Without loss of generality, assume $\vert z'\vert_w\geq \vert z\vert_w$. Let $i\in\N$, we have:
\begin{equation}\label{eq:c non-arch w 2}
\frac{\vert z'^i-z^i\vert_w}{\vert z'-z\vert_w} \leq \vert z'\vert_w^{i-1}
\end{equation}
Equation \eqref{eq:w nmid d first} yields that 
\begin{equation}\label{eq:c non-arch w 3}
\left\vert\frac{B_i(a_1,\ldots,a_d)}{z^iz'}\right\vert_w
\leq \left(\frac{M}{\vert z\vert_w}\right)^{i+1}
\leq \left(\frac{M}{\vert z\vert_w}\right)^{2}
\end{equation}

Equation \eqref{eq:c non-arch w 2} and \eqref{eq:c non-arch w 3}
give:
$$\left|\sum_{i=1}^{\infty}\frac{B_i(a_1,\ldots,a_d)(z'^i-z^i)}{z^iz'^i}
\right|_w\leq \left(\frac{M}{\vert z\vert_w}\right)^{2} \vert z-z'\vert_w < \vert z-z'\vert_w$$
contradicting \eqref{eq:c non-arch w 1}. This finishes the proof 
for the case $w\in M_\Q^0$.

If $w\in M_\Q^0$ corresponds to a prime $p\mid d$, 
then we first note that the exponent of $p$ in $k!$ is
$$\lfloor k/p \rfloor + \lfloor k/p^2\rfloor +\ldots
\leq \frac{k}{p-1}.$$
From Proposition~\ref{prop:deg Ani} and Proposition~\ref{prop:Ani at prime p}, we have:
\begin{align}\label{eq:w mid d first}
\begin{split}
\vert B_{n,j}\vert_w\leq \max_{1\leq k\leq j+1} \vert c_{n,j,k}\vert_w&\leq \max_{1\leq k\leq j+1} \vert d\vert_p^{-nk}\vert d\vert_p^{nk-j-1}p^{k/(p-1)} M^{j+1}\\
&=(Mp^{1/(p-1)}/\vert d\vert_w)^{j+1},
\end{split}
\end{align}
\begin{equation}\label{eq:w mid d second}
\text{and hence } \left\vert\frac{B_{n,j}}{z^j}\right\vert_w \leq 
(Mp^{1/(p-1)}/\vert d\vert_w)\left(\frac{Mp^{1/(p-1)}}{\vert d\vert_w \vert z\vert_w}\right)^{j}.
\end{equation}
We finish the proof using similar arguments as in the case $w\nmid d$
in which equations \eqref{eq:w mid d first} and \eqref{eq:w mid d second} play the role of equations \eqref{eq:w nmid d first}
and \eqref{eq:w nmid d second}.  
\end{proof}


\bigskip
\section{Bounded height in families}
\label{sec:proof}

In this section we complete the proof of Theorem \ref{thm:z^2+t}.  In fact, we prove the more general result of Theorem \ref{thm:general}, relying on the results of Section \ref{sec:Bottcher}.  Throughout this section, let $\cF=\Qbar(t)$. 

\begin{theorem}\label{thm:general}
Let $d\geq 2$, let $f(z)=g(z)=z^d+A_1(t)z^{d-1}+\ldots+A_d(t)\in \Qbar[t][z]$, and let $a,b\in \Qbar$. Assume the following:
\begin{itemize}
	\item [(A)] $d$ is a prime power.	
	\item [(B)] There is a prime $p$ and an embedding
	$\Qbar\rightarrow \C_p$ satisfying the following conditions:
	\begin{itemize}
		\item [(i)] Let $\Zbar_p$ denote the set of elements of 
		$\Qbar_p$ that are integral over $\Z_p$. 
		For every $i$, $A_i(t)\in \Zbar_p[t]$ (in other words,
		$\vert A_i\vert_p\leq 1$) and $\deg(A_i)<i$.
		\item [(ii)] $a\in \Zbar_p$ while $b\notin \Zbar_p$.
		\item [(iii)] For some $m\in\N$, the polynomial
		$f^{m}(a)\in\Zbar_p[t]$ is non-constant and its leading coefficient is a unit. 
	\end{itemize}
\end{itemize} 
Then the pairs $(f,a)$ and $(g,b)$ satisfy the conclusion of Conjecture~\ref{conj:diagonal}.
\end{theorem}

\begin{proof}
Define the set $S$ as in Conjecture~\ref{conj:diagonal}. It
suffices to assume that:
\begin{itemize}
	\item [(C)] the pairs $(f,a)$ and $(f,b)$ are not isotrivial,
	and 
	\item [(D)] for every $m,n\in \N_0$, $f^m(a)\neq f^n(b)$
\end{itemize}
and prove that $S$ has bounded height. Define
$\cM$ and $S_{\cM}$ as in Theorem~\ref{thm:multiplicative dependence}, it suffices to prove that 
$S_\cM$ has bounded height. 
We may assume that $\cM\neq \emptyset$; otherwise there is nothing to prove. 
Let $m_0\in\N$ be minimal such
that the polynomial $f^{m_0}(a)\in\Qbar_p[t]$ 
is non-constant (see condition~B~(iii)); let $\delta_1>0$ denote its degree. From condition (iii) and the form of $f$, we have that $\deg(f^m(a))=d^{m-m_0}\delta_1$ and the leading coefficient of $f^m(a)$ is a unit for
every $m\geq m_0$. Therefore $\hhat_f(a)=\displaystyle\frac{\delta_1}{d^{m_0}}>0$. Since $\cM\neq \emptyset$, we
have $\hhat_f(b)>0$. Hence there is
minimal $n_0\in\N$ such that the polynomial
$f^{n_0}(b)\in\Qbar_p[t]$ is non-constant; let $\delta_2>0$
denote its degree. Then a similar analysis as above yields that $\hhat_f(b)=\frac{\delta_2}{d^{n_0}}$. By the minimality of $m_0$ and $n_0$, along with the form of the polynomial $f$ (see also condition~B~(i)), we have:
\begin{itemize}
\item $1\leq \delta_1,\delta_2<d$. And since $\cM\neq \emptyset$,
we have that $\delta_1/\delta_2$ is a power of $d$. This gives
$\delta_1=\delta_2=:\delta\in \{1,\ldots, d-1\}$.
\item $f^{m_0-1}(a)\in\Qbar\subset \Qbar_p$ and $\vert f^{m_0-1}(a)\vert_p\leq 1$.
\item $f^{n_0-1}(b)\in\Qbar\subset \Qbar_p$ and $\vert f^{n_0-1}(b)\vert_p> 1$ (note that $|b|_p>1$ while each $A_i\in\Zbar_p[t]$).
\end{itemize}
Therefore, after replacing $(a,b)$ by 
$(f^{m_0-1}(a),f^{n_0-1}(b))$, from now on, we assume
that $m_0=n_0=1$. We need to prove that the set
$$S_{\cM}=\{t\in \Qbar: f^n_t(a)=g^n_t(b)\ \text{for some $n\geq 0$}\}$$
has bounded height.

Let $K$ be a number field such that 
$a,b\in K$ and the polynomials $A_i$ belong to $K[t]$.
As in the discussion of \S\ref{strategy}, applying Corollaries \ref{cor:roots} and \ref{cor:height of polynomials}, it suffices to prove the following claim.

\medskip
\textbf{Claim:} There exists a positive constant $c_{19}$ depending on $K$, $p$, $f$, $a$, and $b$ such that the following holds. For every $t_0\in S_{\cM}$, let $N\in\N$ be minimal such that $f^N_{t_0}(a)=f^N_{t_0}(b)$; then we have $[K(t_0):K]\geq c_{19}d^N$. 

\medskip
We prove this claim as follows.  Fix a positive integer $c_{18}$ such that
$\vert b\vert_p^{\delta d^{c_{18}-1}-d}>p^{1/(p-1)}/\vert d\vert_p$.  Fix $t_0\in S_{\cM}$ and let $N$ be as in the claim; we may assume $N>c_{18}$.  First, we observe that 
$\vert t_0\vert_p\ge \vert b\vert_p$. Otherwise, we would have
$$\vert f^N_{t_0}(a)\vert_p \leq \max\{1,\vert t_0\vert_p\}^{\delta\cdot d^{N-1}}<|b|_p^{d^N},$$
while $\vert f^N_{t_0}(b)\vert_p=\vert b\vert_p^{d^N}$, contradiction.

Now since $\vert t_0\vert_p\geq \vert b\vert_p>1$ and 
$f^n(a)\in\Zbar_p[t]$ is a polynomial of
degree $\delta d^{n-1}$ whose leading coefficient is a unit, 
we have: 
\begin{equation}\label{eq:f^nat_0p}
\vert f^n_{t_0}(a)\vert_p=\vert t_0\vert_p^{\delta d^{n-1}}\ \text{for every $n\in\N$.} 
\end{equation}

On the other hand, let $n_1\geq 0$ be minimal such that
$\vert f^{n_1}_{t_0}(b)\vert_p \geq \vert t_0\vert_p$; note that $n_1$ exists since $\vert f^{N}_{t_0}(b)\vert_p=\vert f^N_{t_0}(a)\vert_p =\vert t_0\vert_p^{\delta d^{N-1}}\geq \vert t_0\vert_p$.
From condition (i), we have:
\begin{equation}\label{eq:f^nbt_0p}
\vert f^n_{t_0}(b)\vert_p=\vert f^{n_1}_{t_0}(b)\vert_p^{d^{n-n_1}}\ \text{for every $n\geq n_1$.}
\end{equation}

From \eqref{eq:f^nat_0p}, \eqref{eq:f^nbt_0p},
and $f^N_{t_0}(a)=f^N_{t_0}(b)$, we have:
\begin{equation}\label{eq:compare}
\vert t_0\vert_p^{\delta d^{n_1-1}}=\vert f^{n_1}_{t_0}(b)\vert_p.
\end{equation}
If $n_1=0$, equation 
\eqref{eq:compare} would give: 
$$|b|_p=\vert f^{0}_{t_0}(b)\vert_p=|t_0|_p^{\delta/d}<\vert t_0\vert_p,$$  contradicting the earlier observation that
$\vert t_0\vert_p \geq \vert b\vert_p$. Hence
$n_1\geq 1$. We show next that $n_1=1$. Indeed, if $n_1\geq 2$ then
by using $\vert f^{n_1-1}_{t_0}(b)\vert_p <\vert t_0\vert_p$ 
due to the minimality of $n_1$
and by induction, we get that $$\vert f^n_{t_0}(b)\vert_p <\vert t_0\vert_p^{d^{n-n_1+1}}\leq 
\vert t_0\vert_p^{d^{n-1}}$$
for every $n\geq n_1-1$. In particular, when $n=N$, we have:
$$\vert f^N_{t_0}(a)\vert_p=\vert f^N_{t_0}(b)\vert_p<\vert t_0\vert_p^{d^{N-1}}$$
contradicting \eqref{eq:f^nat_0p}. Therefore $n_1=1$.

Let 
	$$\cD':=\left\{(z,t)\in \C_p^2:\ \displaystyle\frac{\max\{1,\vert A_1(t)\vert_p,\ldots,
\vert A_d(t)\vert_p\}}{\vert d\vert_p}p^{1/(p-1)}<\vert z\vert_p\right\}.$$
From part (i) of condition~(B), we have $\vert A_i(t_0)\vert_p<\vert t_0\vert_p^d$. From the fact that $n_1=1$ coupled with equations \eqref{eq:f^nat_0p}, \eqref{eq:f^nbt_0p}, \eqref{eq:compare}, along with the choice of $c_{18}$, 
and 
the inequality
$\vert t_0\vert_p\geq \vert b\vert_p$, we have:
$$\vert f^{c_{18}}_{t_0}(a)\vert_p=\vert t_0\vert_p^{\delta d^{c_{18}-1}}>\vert t_0\vert_p^d p^{1/(p-1)}/\vert d\vert_p$$
and
$$\vert f^{c_{18}}_{t_0}(b)\vert_p=\vert t_0\vert_p^{\delta d^{c_{18}-1}}>\vert t_0\vert_p^d p^{1/(p-1)}/\vert d\vert_p.$$
Therefore $(f^{c_{18}}_{t_0}(a),t_0)$ and
$(f^{c_{18}}_{t_0}(b),t_0)$
belong to $\cD'$.
Let $\tilde{\cB}$ be the function in Theorem~\ref{thm:Bottcher}
and define $\tilde{\cB}'(z,t)=\tilde{\cB}(z,A_1(t),\ldots,A_d(t))$
which is well-defined on $\cD'$ thanks to the definition of $\cD'$
and Theorem~\ref{thm:Bottcher} (regardless of whether $p\mid d$
or not). From $f^N_{t_0}(a)=f^N_{t_0}(b)$
and the functional equation of $\cB$ in Theorem~\ref{thm:Bottcher},
we have:
$$\tilde{\cB}'(f^{c_{18}}_{t_0}(a),t_0)^{d^{N-c_{18}}}=
\tilde{\cB}'(f^{c_{18}}_{t_0}(b),t_0)^{d^{N-c_{18}}}.$$

In other words, we have
$\zeta:=\displaystyle\frac{\tilde{\cB}'(f^{c_{18}}_{t_0}(a),t_0)}{\tilde{\cB}'(f^{c_{18}}_{t_0}(b),t_0)}$
is a $d^{N-c_{18}}$-th root of unity.
On the other hand, if the order of $\zeta$ divides 
$d^{N-c_{18}-1}$ then we have:
$$\tilde{\cB'}(f^{c_{18}}_{t_0}(a),t_0)^{d^{N-c_{18}-1}}=
\tilde{\cB'}(f^{c_{18}}_{t_0}(b),t_0)^{d^{N-c_{18}-1}}$$
which gives 
$$\tilde{\cB'}(f^{N-1}_{t_0}(a),t_0)=
\tilde{\cB'}(f^{N-1}_{t_0}(b),t_0)$$
thanks to the functional equation satisfied by $\tilde{\cB}$.
By Theorem~\ref{thm:Bottcher}, we have $f^{N-1}_{t_0}(a)=f^{N-1}_{t_0}(b)$ contradicting the minimality of $N$. 

Write $\kappa=K_p(t_0)$ where $K_p\subset \Qbar_p$ is the completion of $K$ under $\vert\cdot\vert_p$. We have proved that the field $\kappa$ contains 
a $d^{N-c_{18}}$-th root of unity $\zeta$
and the order of $\zeta$ does not divide $d^{N-c_{18}-1}$.
This is the only place where we use the technical assumption that $d$ is a prime power; we conclude that the order of
$\zeta$ is a strict multiple of
$d^{N-c_{18}-1}$ and hence, see \cite[pp.~158--159]{Neu99}, we have:
$$[K_p(\zeta):K_p]\geq c_{20}d^{N-c_{18}}$$
for some constant $c_{20}$ that depends only on $K_p$ and 
$d$. Let $c_{19}=c_{20}d^{-c_{18}}$, we have:
$$[K(t_0):K]\geq [\kappa:K_p]\geq [K_p(\zeta):K_p]\geq c_{19}d^N$$
and this proves the claim. Then Corollary~\ref{cor:height of polynomials} (along with Corollary~\ref{cor:roots}) allows us to conclude the proof of Theorem~\ref{thm:general}.
\end{proof}

We have the following immediate corollary, which is itself a generalization of Theorem~\ref{thm:z^2+t}.  
\begin{cor}\label{cor:general}
Let $d$ be a prime power and let $f(z)=z^d+t\in\Qbar[t][z]$. 
Let $a,b\in \Qbar$ exactly one of which is an algebraic integer. 
Then the set 
$$S=\{t_0\in\Qbar:\ f^m_{t_0}(a)=f^n_{t_0}(b)\ \text{for some $m,n\in\N$}\}$$
has bounded height.
\end{cor}
\begin{proof}
We can easily check that $(f,a)$ and $(f,b)$ are not isotrivial.
Without loss of generality, assume that $a$ is an algebraic integer while $b$ is not. There is a prime number $p$ such that, under a suitable embedding $\Qbar\rightarrow \Qbar_p$, 
$b$ is not integral over $\Z_p$. We have that $f^m(a)\in\Z_p[t]$
while $f^n(b)\notin \Z_p[t]$, hence $f^m(a)\neq f^n(b)$ for every
$m,n\in\N$. We apply Theorem~\ref{thm:general} and get
the bounded height result.
\end{proof}

It is an interesting problem to remove the technical condition that $d$ is a prime power in Theorem~\ref{thm:general}. 
This condition 
is only used at the end of the proof of Theorem~\ref{thm:general}
in order to show that the order of $\zeta$ is comparable
to $d^N$. 



\end{document}